\documentclass[11pt]{amsart}
\usepackage{amssymb,latexsym,comment,url,amsmath}

\usepackage[T1]{fontenc}
\usepackage{cite}

\newcommand{\EE}{{\mathcal E}}

\newcommand{\Q}{{\mathbb Q}}

\newcommand{\rank}{\operatorname{rank}}

\newcommand{\Z}{{\mathbb Z}}

\newcommand{\Magma}{{\sf MAGMA }}

\newenvironment{Proof}{\par\noindent{\sc Proof:}}%
                      {\hspace*{\fill}\nobreak$\Box$\par\medskip}
                       {\hspace*{\fill}\nobreak$\Box$\par\medskip}

\newtheorem{Proposition}{Proposition}[section]
\newtheorem{Theorem}[Proposition]{Theorem}
\newtheorem{Lemma}[Proposition]{Lemma}

\newtheorem{Corollary}[Proposition]{Corollary}

\theoremstyle{definition}

\newtheorem{Remark}[Proposition]{Remark}
\newtheorem{example}{Example}

\renewcommand{\baselinestretch}{1.1}

\begin{document}
\title[Rational Points in Geometric Progression on the Unit Circle]
{Rational Points in Geometric Progression on the Unit Circle}
\author[G. S. \c{C}elik, M. Sadek and G. Soydan]{Gamze Sava\c{s} \c{C}elik, Mohammad Sadek and G\"{o}khan Soydan}

\address{{\bf Gamze Sava\c{s} \c{C}elik}\\
Department of Mathematics \\
Bursa Uluda\u{g} University\\
G\"{o}r\"{u}kle Campus, 16059 Bursa, Turkey}
	
\email{gamzesavascelik@gmail.com}

\address{{\bf Mohammad Sadek}\\
Faculty of Engineering and Natural Sciences\\
Sabanc{\i} University\\34956 Tuzla, \.{I}stanbul\\Turkey}
	
\email{mmsadek@sabanciuniv.edu}

\address{{\bf G\"{o}khan Soydan} \\
Department of Mathematics \\
Bursa Uluda\u{g} University\\
G\"{o}r\"{u}kle Campus, 16059 Bursa, Turkey}
\email{gsoydan@uludag.edu.tr }

\newcommand{\acr}{\newline\indent}
	

\maketitle
	
\begin{abstract}
  A sequence of rational points on an algebraic planar curve is said to form an $r$-geometric progression sequence if either the abscissae or the ordinates of these points form a geometric progression sequence with ratio $r$. In this work, we prove the existence of infinitely many rational numbers $r$ such that for each $r$ there exist infinitely many $r$-geometric progression sequences on the unit circle $x^2 + y^2 = 1$ of length at least $3$. 	
\end{abstract}

\bigskip

\section{Introduction}
\label{intro}
In the past 30 years, several mathematicians have been interested in whether the $x$-coordinates (or $y$-coordinates) of rational points on various familes of algebraic plane curves over the rational field $\Q$ form a sequence of rational numbers which is either an arithmetic or a geometric progression sequence. In 1992, Lee and V\'{e}lez, \cite{LV}, proved that for infinitely many rational values $a$ the curve $y^2 = x^3+a$ possesses a sequence of rational points whose $x$-coordinates form an arithmetic progression sequence of length at least 4. Seven years later, Bremner, \cite{Br}, discussed arithmetic progression sequences on other families of elliptic curves over $\Q$. He investigated the size of these sequences and produced elliptic curves with 7-term and 8-term arithmetic progression sequences. The problem of finding rational points in arithmetic progression on elliptic curves, hyperelliptic curves and different models of elliptic curves including Edwards and Huff curves has been explored extensively, \cite{Al,BrE,Cam,Ch,Mac,Mod1,Mod3,ST,U1,U2}.

In 2013, the first work on geometric progression sequences was introduced by Bremner and Ulas, \cite{BU}. For any nontrivial four term geometric progression $x_i$, $i=1,2,3,4$, they first proved that there exist infinitely many pairwise non-isomorphic hyperelliptic curves $C: y^2=ax^n+b$ such that $x_i$ is the $x$-coordinate of a rational point on $C$. Secondly, they showed that there exist infinitely many parabolas $y^2=ax+b$ which contain five points in geometric progression. Three years later, Alaa and Sadek, \cite{AS}, found a family of hyperelliptic curves defined by the equation $y^2=ax^{2n}+bx^n+a$, $a,b\in\mathbb{Z}$, such that this family posseses a geometric progression sequence of rational points whose length is at least 8. In 2017, Ciss and Moody, \cite{Mod2}, considered long geometric progressions on different models of elliptic curves, including Weierstrass curves, Edwards curves, Huff curves and quartic curves. They presented infinite families of (twisted) Edwards curves, and Huff curves containing five rational points in geometric progression, an infinite family of Weierstrass curves containing eight points in progression, as well as an infinite family of quartic curves with 10-term progression sequences. Given an arbitrary sequence of rational points of length 4, 5, or 6, \c{C}elik, Sadek and Soydan, \cite{CSS}, provided explicit examples of (twisted) Edwards curves and (general) Huff curves on which this sequence is realized as the $x$-coordinated of rational points.

Arithmetic and geometric progression sequences have also been studied on conics. In 2013, Alvarado and Goins, \cite{AG}, gave a generalization of 3-term arithmetic progressions on an arbitrary conic section. Three years later, Choudhry and Juyal, \cite{Ch2}, parameterized infinitely many arithmetic progressions of three rational points on the unit circle, such that the three points lie in the first quadrant. They also used these progressions to derive infinitely many arithmetic progressions on the ellipse $x^2/a^2+y^2/b^2=1$. Recently, Ciss and Moody, \cite{Mod4}, considered long arithmetic progressions on conics. They gave a slightly more general result on 3-term arithmetic progressions on the unit circle $x^2+y^2=1$, and similarly on the unit hyperbola $x^2-y^2=1$. They also provided infinitely many conics $ax^2+cy^2=1$ that have 8-term arithmetic progression sequences.

In the present work, we investigate sequences of rational points in geometric progression on the unit circle. We prove the existence of infinitely many rational numbers $r$ such that for each $r$, there exist geometric progression sequences of rational points on the unit circle $ x^2+y^2=1$ with ratio $r$ and length at least $3$. This will be accomplished by analyzing the arithmetic of several rational elliptic surfaces. In addition, we prove that there are infinitely many rational points $P$ on the unit circle such that $P$ is a term in infinitely many geometric progression sequences of length at least 3.

\bigskip
\section{Geometric progression sequences of length $2$ on the unit circle}\label{sec:2}

Let $C$ be the unit circle $x^2+y^2=1$. The set $C(K)$ of $K$-rational points on $C$ is defined by $C(\Q) = \{(x,y) : x^2+y^2=1, x,y \in \Q\}$.

A sequence of rational points $(x_1,y_1),\ldots,(x_n,y_n)$ in $C(\Q)$ will be said to form an $r$-geometric progression sequence of length $n$ if $x_{i}/x_{i-1}=r$ for every $i=2,\ldots,n$.

\begin{Lemma}
\label{lem1}
  Let $\displaystyle r=\frac{-4m}{m^2+2}$ where $m\in\Q$. There exists infinitely many pairs of rational points $(x_1,y_1),(x_2,y_2)\in C(\Q)$ such that $x_2/x_1=r$. In particular, there are infinitely many $r$-geometric progression sequences of length $2$.
\end{Lemma}
\begin{Proof}
Let $(x_1,y_1)$ and $(rx_1,y_2)$ be such that $x_1^2+y_1^2=1$ and $(rx_1)^2+y_2^2=1$. We consider the following intersection of two quadratic surfaces in $\mathbb{P}^3$
\[H_r:x_1^2+y_1^2=z^2,\qquad r^2x_1^2+y_2^2=z^2.\] One may check using \Magma \cite{Magma} that the latter curve is an elliptic curve that may be described by the following Weierstrass equation
\[E_r:y^2=x(x-4)(x-4 r^2).\] A full description of the birational isomorphism $\phi_r:E_r\to H_r$ between the two curves can be found for example in \cite[Theorem 3.1]{An}.
Choosing $x=2r^2$ yields that $y^2=-8(r^2-2)r^4$. If $(r_0,s_0)$ is a point on the conic $Q:s^2+2r^2=4$, then $P_0=(x(P_0),y(P_0))=(2r_0^2,2r_0^2s_0)\in E_{r_0}(\Q)$ is a point of infinite order. Since the conic $Q$ possesses one rational point $(r,s)=(0,2)$, it contains infinitely many rational points whose parametrization is given by $(r,s)=\left(-4m/(m^2+2),2 ( m^2-2)/( m^2+2)\right)$. Setting $mP_0$ to be the $m$-th multiple of $P_0$, then $Q_m:=\phi_r(mP_0):=(x_1,y_1,y_2,z)\in H_r(\Q)$ where $x_1$ and $x_2=rx_1$ are the $x$-coordinates of rational points forming an $r$-geometric sequence in $C(\Q)$. This proves the statement.
\end{Proof}
In fact, the lemma above can be strengthened as follows:
\begin{Proposition}
\label{prop1}
 There are infinitely many rational numbers $\displaystyle r$ such that for each such $r$,
  there exists infinitely many pairs of rational points $(x_1,y_1),(x_2,y_2)\in C(\Q)$ with $x_2/x_1=r^2$.
\end{Proposition}
\begin{Proof}
To guarantee the existence of such rational number $r$, one has \begin{eqnarray*}\label{eq1}x_1=\frac{2s}{1+s^2},\qquad x_2=\frac{2t}{1+t^2}=r^2\frac{2s}{1+s^2}.
\end{eqnarray*}
In other words, the existence of such pair yields a rational point on the following planar curve
\[\mathcal{E}_{r}:t(s^2+1)=r^2s(t^2+1).\]
The curve $\mathcal{E}_{r}$ is a twisted Huff curve, see \cite{Joye}. In particular, the above equation describes an elliptic curve whose Weierstrass equation is \[E'_{r}:y^2=x(x-1)(x-r^4).\]
The curves $\mathcal{E}_r$ and $E'_r$ are birationally isomorphic via the following transformations \[(x,y)=\left(-r^2(r^2t-s)/(-t+r^2s),-r^2(r^4-1)/(-t+r^2s)\right),\qquad  (s,t)=\left((x-r^4)/y,r^2(x-1)/y\right),\] see \cite[\S 3.3]{Joye} for details. The above Weierstrass equation $E'_r$ is a Legendre equation, and the torsion group of the curve contains $\Z/2\Z\times\Z/2\Z$.

 Fix $u\in\Q\setminus\{0,1\}$. Setting $x=u^4$,
we have $y^2=u^4(u^4-1)(u^4-r^4).$ Setting $w=y/u^2$, one obtains the quartic elliptic curve $$H: w^2 = - (u^4-1)r^4+(u^4-1)u^4.$$
The point $P=(r,w)=(1,u^4-1)$ is a point of infinite order on the curve $H$. In conclusion, there are infinitely many $r$ such that  $$(x,y)=(u^4, u^2(u^4-1))\in E'_{r}(\Q)\qquad\textrm{is a point of infinite order}.$$
More precisely, if we choose $r$ to be the $r$-coordinate of $mP$, $m\ne \pm 1$, in $H(\Q)$, then $\mathcal{E}_{r}$ is a curve of positive rank. A point of infinite order may be found on the latter curve by using the transformation in \cite{Joye} to find the image of $(u^4, u^2(u^4-1))$ in $\mathcal{E}_{r}(\Q)$.
\end{Proof}
\begin{example}
In this example we produce three pairs of rational points on the unit circle that form $r^2$-geometric progression sequences when $r=5/4$. For $(s,t)=(8,1/5)$, we find the points $(x_1,y_1)=(16/65,63/65))$ and $(x_2,y_2)=(5/13,12/13)$. For $(s,t)=(64/273,21/52)$, we obtain the pair $(34944/78625,70433/78625)$ and $(2184/3145,2263/3145)$. When $(s,t)=(37523/119144,67159/41605)$, we get the following pair of rational points $(8941280624/15603268265,12787317207/15603268265)$, $(2794150195/3120653653,1389677628/3120653653)$ on the unit circle.

We recall that the three points $(s,t)$ above are rational points on the curve $\EE_{5/4}$ in the proof of Proposition \ref{prop1}. Moreover, the images of these three rational points in the elliptic curve described by the Weierstrass equation $E'_{5/4}$ are $(125/128, 375/2048)$, $(4225/256,61425/1024)$ and $(351125/114242, 876825375/436861408)$, respectively.
\end{example}	
\begin{Remark}
The rational values $r$ are parametrized by rational points on a conic in Lemma \ref{lem1}, whereas the rational values $r$ are parametrized by rational points on an elliptic curve of positive rank in Proposition \ref{prop1}.
\end{Remark}

\section{Geometric progression sequences of length $3$ on the unit circle}\label{sec:3}

Assume that the points
$$(x_1,y_1)=(u/r,f),\,\,(x_2,y_2)=(u,g),\,\,(x_3,y_3)=(ur,h)$$
are in $C(\Q)$ where $C:x^2+y^2=1$. Set

\begin{equation}\label{main-par}
(u/r,f)=\left(\dfrac{2t}{1+t^2},\dfrac{1-t^2}{1+t^2}\right),\,\,(u,g)=\left(\dfrac{2s}{1+s^2},\dfrac{1-s^2}{1+s^2}\right),
\end{equation}
so that $r=\dfrac{2s}{1+s^2}\cdot \dfrac{1+t^2}{2t}$. Now, the fact that $ (x_3,y_3)\in C(\Q)$ is equivalent to the following condition
$$h^2=1-\left(\dfrac{2s}{1+s^2}\right)^4 \cdot \left(\dfrac{1+t^2}{2t}\right)^2.$$
Equivalently, one has
\begin{equation}\label{EC1-4}
t^2(s^2+1)^4-4s^4(t^2+1)^2=H^2,
\end{equation}
where $H=h(s^2+1)^2t.$
The curve \eqref{EC1-4} is a quartic elliptic curve over $\Q(s)$, with the $\Q(s)$-rational point $(t,H)=(s,s^5-s)$. It can be described by the Weierstrass equation
\begin{equation}\label{EC1-3}
E_s: y^2=x(x+16s^4)(x+(1+s^2)^4),
\end{equation}
and its torsion group is $\Z/2\Z\times \Z/4\Z$, see for example \cite[\S 2]{Dujella}.

Now we give the main result:
\begin{Theorem}
\label{thm1}
There are infinitely many rational numbers $r$ such that for each $r$ there is an $r$-geometric progression sequence of length at least $3$ on the unit circle $x^2+y^2=1$.
\end{Theorem}
\begin{proof}
The above discussion asserts that one only needs to find infinitely many rational values for $s$ such that $E_s$ has positive rank over $\Q$.

If $x=8s^3(1+s^2)$ is the $x$-coordinate of a rational point on $E_s$, then $s^4-2s^3+6s^2-2s+1$ must be a rational square. In other words, one will get a rational point on the quartic elliptic curve
\begin{equation}\label{EC2-4}
w^2=s^4-2s^3+6s^2-2s+1
\end{equation}	
which possesses the $\Q$-rational point $(s,w)=(0,\pm 1)$. Using the procedure in \cite[chapter 3, pp. 89-91]{JC}, we see that the latter quartic curve is birationally equivalent to the elliptic curve defined by the Weierstrass equation
\begin{equation}\label{EC2-3}
G: v^2=u^3-972u
\end{equation}	
and the point $(u,v)=(-27,81)\in G(\Q)$ is a point of infinite order. In fact, one can check using MAGMA \cite{Magma} that $\rank(G(\Q))=1$. This completes the proof.
\end{proof}
The following corollary follows from the proof above.
\begin{Corollary}
There are infinitely many rational numbers $s$ such that the point $(2s/(1+s^2),(1-s^2)$ $/(1+s^2))\in C(\Q)$ lies in infinitely many geometric progression sequences on $C$ of length at least $3$.
\end{Corollary}

In the following table we give examples of geometric progression sequences of length at least $3$ in $C(\Q)$.

\begin{table}[h]
	\centering
	\begin{tabular}{|c|c|}
		\hline
		$(r,s,t)$ & $(x_1,x_2,x_3)$\\\hline
		(39/25, 3, 5) & (5/13, 3/5, 117/125) \\\hline
		(6409/3034, 4, 328/37)& (24272/108953, 8/17, 1508/1517)\\\hline
		(5987825/3616561, 5, 1537/181) & (278197/1197565, 5/13, 29939125/47015293)\\\hline
		(55045/24531, 6, 234/17) & (7956/55045, 12/37, 220180/302549)\\\hline
		(7935762913/2225017375, 7, 125885/4949)& (623004865/7935762913, 7/25, 7935762913/7946490625)\\\hline
		(6548713889/6051759025, 8, 80392/9265)& (1489663760/6548713889, 16/65, 104779422224/393364336625)\\\hline
	\end{tabular}
	\vspace{0.2cm} \caption{}
\end{table}

One notices that if $ur^2$ is the $x$-coordinate of a rational point on the unit circle in Theorem \ref{thm1}, then this implies the existence of an $r$-geometric progression sequence of length at least 4. However, this will require studying rational solutions of the following system of Diophantine equations:
\begin{eqnarray}\label{intersection}
t^2(s^2+1)^4-4s^4(t^2+1)^2&=&H_1^2\nonumber\\
t^4(s^2+1)^6-4s^6(t^2+1)^4&=&H_2^2.\nonumber
\end{eqnarray}
The authors intend to study the above system of equations. Moreover, they will tackle the question of the existence of geometric progression sequences on planar conics in future work.

\section{Acknowledgements}
We are grateful to the anonymous referee for many valuable suggestions and useful comments. We would also like to thank Professors Andrew Bremner, Andrej Dujella and Allan J. Macleod for their usefull ideas.  The first and third authors were supported by the Research Fund of Bursa Uluda\u{g} University under Project No: F-2020/8.


\begin{thebibliography}{MM}
\frenchspacing
\renewcommand{\baselinestretch}{1}

\bibitem{AS} {\sc M. Alaa, M. Sadek}, On geometric progressions on hyperelliptic curves, {\em J. Integer Seq.} {\bf 19} (2016), Article 16.6.3.

\bibitem{Al} {\sc A. Alvarado}, An arithmetic progression on quintic curves, {\em ibid.} {\bf 12} (2009), Article 09.7.3.

\bibitem{AG} {\sc A. Alvarado, E. H. Goins}, Arithmetic progressions on conic sections. {\em Int J Number Theory.} {\bf 9} (2013), 1379--1393.

\bibitem{An} {\sc S.Y. An, S.Y. Kim, D.C. Marshall, S.H. Marshall, W.G. McCallum, A.R. Perlis},
Jacobians of genus one curves, {\em Journal of Number Theory} {\bf 90} (2001), no. 2, 304--315.

\bibitem{Magma} {\sc W. Bosma, J.  Cannon, C.  Playoust}, The Magma Algebra System I. The user language, {\sc J. Symbolic Comput.} {\bf 24} (1997), no. 3-4, 235--265.

\bibitem{Br} {\sc A. Bremner}, On arithmetic progressions on elliptic curves, {\em Experiment Math.} {\bf 8} (1999), 409--413.

\bibitem{BrE} {\sc A. Bremner}, Arithmetic progressions on Edwards curves, {\em J. Integer Seq.} {\bf 16 } (2013), Article 13.8.5.

\bibitem{BU} {\sc A. Bremner, M. Ulas}, Rational points in geometric progressions on certain hyperelliptic curves, {\em Publ. Math. Deb.} {\bf 82} (2013), 669--683.

\bibitem{Cam} {\sc G. Campbell}, A note on arithmetic progressions on elliptic curves, {\em J. Integer Seq.} {\bf 6} (2003), Article 03.1.3.

\bibitem{Ch} {\sc A. Choudhry}, Arithmetic progressions on Huff curves, {\em ibid.} {\bf 18} (2015), Article 15.5.2.

\bibitem{Ch2} {\sc A. Choudhry, A. Juyal}, Rational points in arithmetic progression on the unit circle, {\em ibid.} {\bf 19} (2016), Article 16.4.1.

\bibitem{Mod4}{\sc A. A. Ciss, D. Moody}, Arithmetic progressions on conics, {\em ibid.} {\bf 20} (2017), Article 17.2.6.


\bibitem{Mod2}{\sc A. A. Ciss, D. Moody}, Geometric progressions on elliptic curves, {\em Glasnik Math.} {\bf 52} (2017), 1--10.



\bibitem{JC} {\sc J. E. Cremona}, Algorithms for Modular Elliptic Curves, Cambridge University Press., $2$ nd edu., (1997).

\bibitem{CSS} {\sc G. S. \c{C}elik, M. Sadek and G. Soydan}, Rational sequences on different models of elliptic curves, {\em Glasnik Math.} {\bf 54} (2019), 53--64.

\bibitem{Dujella} {\sc A. Dujella, J. C. Peral}, High-rank elliptic curves with torsion $\Z/2\Z \times\Z/4\Z$ induced by Diophantine triples, {\em LMS J. Comput. Math.} {\bf 17} (1) (2014) 282--288.

\bibitem{Joye} {\sc M. Joye, M. Tibbouchi, D. Vergnaud}, Huff's Model for Elliptic Curves. {\em Algorithmic Number Theory-ANTS-IX, Lecture Notes in Computer Science}, {\bf 6197} (2010), Springer, 234--250.

\bibitem{LV} {\sc J. B. Lee, W. Y. V\'{e}lez}, Integral solutions in arithmetic progression for $y^{2}=x^{3}+k$, {\em Per. Math. Hung.} {\bf 25} (1992), 31--49.

\bibitem{Mac} {\sc A. J. Macleod}, $14$-term arithmetic progressions on quartic elliptic curves, {\em J. Integer Seq.} {\bf 9} (2006), Article 06.1.2.

\bibitem{Mod1}{\sc D. Moody}, Arithmetic progressions on Huff curves, {\em Ann. Math. Inform.} {\bf 38} (2011), 111--116.

\bibitem{Mod3}{\sc D. Moody}, Arithmetic progressions on Edwards curves, {\em J. Integer Seq.} {\bf 14} (2011), Article 11.1.7.

\bibitem{ST} {\sc I. Garcia-Selfa, J. Tornero}, Searching for simultaneous arithmetic progressions on elliptic curves, {\em Bull. Austral. Math. Soc.} {\bf 71} (2005), 417--424.


\bibitem{U1} {\sc M. Ulas}, A note on arithmetic progressions on quartic elliptic curves, {\em J. Integer Seq.} {\bf 8} (2005), Article 05.3.1.


\bibitem{U2} {\sc M. Ulas}, On arithmetic progressions on genus two curves, {\em Rocky Mountain J. Math.} {\bf 39} (2009), 971--980.



\end{thebibliography}
\end{document}